\documentclass[a4paper,12pt]{article}
\usepackage{amsmath,amsfonts,enumerate,amssymb,amsthm,color}
\usepackage{pgf,tikz,cite}
\usepackage{authblk}
\usepackage{fullpage,enumerate}

\newtheorem{thm}{Theorem}
\newtheorem{lem}{Lemma}

\newtheorem{prop}{Proposition}
\newtheorem{cor}{Corollary}
\newtheorem{problem}{Problem}

\newtheorem{defi}{Definition}

\newcommand{\comment}[1]{}

\newcommand{\R}{{\mathbb R}}

\begin{document}

\title{Equistarable graphs and counterexamples to three conjectures on equistable graphs\thanks{M.M.~is supported in part by ``Agencija za raziskovalno dejavnost Republike Slovenije'', research program P$1$--$0285$ and research projects J$1$--$4010$, J$1$--$4021$, J$1$--$5433$, and N$1$--$0011$: GReGAS, supported in part by the European Science Foundation.
N.T.~is partially supported by ANR project Stint under reference ANR-13-BS02-0007.
Part of this research was carried out during the visit of M.M.~to N.T.~at ENS Lyon; their support is gratefully acknowledged.}}

\author[1,2,3]{Martin Milani\v c}
\author[4]{Nicolas Trotignon}

\affil[1]{\scriptsize{}University of Primorska, UP IAM}
\affil[ ]{Muzejski trg 2, SI-6000 Koper, Slovenia, \texttt{martin.milanic@upr.si}}
\affil[ ]{}
\affil[2]{University of Primorska, UP FAMNIT} 
\affil[ ]{Glagolja\v{s}ka 8, SI-6000 Koper, Slovenia}
\affil[ ]{}
\affil[3]{IMFM, Jadranska 19, 1000 Ljubljana, Slovenia}
\affil[ ]{}
\affil[4]{CNRS, LIP, ENS de Lyon, INRIA, Universit\'e de Lyon}
\affil[ ]{\texttt{$\{$nicolas.trotignon$\}$@ens-lyon.fr}}

\maketitle

\begin{abstract}
Equistable graphs are graphs admitting positive weights on vertices such
that a subset of vertices is a maximal stable set if and only if it is of total
weight $1$. In $1994$, Mahadev et al.~introduced a subclass of equistable graphs, called strongly equistable graphs,
as graphs such that for every $c \le 1$
and every non-empty subset $T$ of vertices that is not a maximal stable
set, there exist positive vertex weights such that every maximal stable set
is of total weight $1$ and the total weight of $T$ does not equal $c$.
Mahadev et al.~conjectured that every equistable graph is strongly equistable.
General partition graphs are the intersection graphs of set systems over a finite
ground set $U$ such that every maximal stable set of the graph corresponds
to a partition of $U$. In $2009$, Orlin proved that every general partition graph is equistable, and
conjectured that the converse holds as well.

Orlin's conjecture, if true, would imply the conjecture due to Mahadev, Peled, and Sun. An intermediate
conjecture, one that would follow from Orlin's conjecture and
would imply the conjecture by Mahadev, Peled, and Sun, was posed by
Miklavi\v c and Milani\v c in $2011$, and states that every equistable graph has a
clique intersecting all maximal stable sets.
The above conjectures have been verified for several graph
classes. We introduce the notion of equistarable graphs and based on it construct
counterexamples to all three conjectures within the class of complements
of line graphs of triangle-free graphs.
\end{abstract}


\section{Introduction}

We consider finite simple undirected graphs. A {\it stable} (or {\it independent}) set in a graph is a set of pairwise non-adjacent vertices; a stable set
is said to be {\it maximal} if it is not contained in any other stable set. In 1980, Payan introduced equistable graphs as a generalization of threshold graphs. A graph $G=(V,E)$ is said to be {\it equistable} if and only if there exists a mapping $\varphi:V\to \R_+$ such that
for all $S\subseteq V$,
$S \textrm{\ is\ a\ maximal\ stable\ set\ of\ } G$ if and only if
\hbox{$\varphi(S):= \sum_{v\in S}\varphi(v) = 1\,.$}~\cite{MR553649}.
The mapping $\varphi$ is called an {\it equistable weight function} of $G$.
Equistable graphs were studied in a series
of papers~\cite{MR2794315,MR2379078,MR3162288,MR2024271,MR2024275,MR2823204,MR1268776,MR553649,AlcGutKovMilRiz14,BorGurMil14}, however, the complexity status of recognizing equistable graphs is open, and no combinatorial characterization of equistable graphs is known.

In $1994$, Mahadev et al.~introduced a subclass of equistable graphs, the so-called strongly equistable graphs~\cite{MR1268776}.
For a graph $G$, we denote by ${\cal S}(G)$ the set of all maximal stable sets of $G$, and by ${\cal T}(G)$ the set of all other
nonempty subsets of $V(G)$.
A graph $G=(V,E)$ is said to be {\it strongly equistable} if for each $T\in {\cal T}(G)$ and each $\gamma\le 1$ there exists a
function $\varphi:V\to \R_+$ such that  $\varphi(S) = 1{\rm\ for\ all\ } S\in {\cal S}(G)$, and $\varphi(T)\neq \gamma$. Mahadev et al.~showed that
every strongly equistable graph is equistable, and conjectured that the converse assertion is valid.
Partial results on the conjecture are known. The conjecture is known to hold for a class of graphs
containing all perfect graphs~\cite{MR1268776}, for series-parallel graphs~\cite{MR2024271}, for line graphs~\cite{MR3162288} and
more generally for EPT graphs~\cite{AlcGutKovMilRiz14}, for AT-free graphs~\cite{MR2794315}, and for various product graphs~\cite{MR2794315}.

A combinatorially defined graph class closely related to equistable graphs is the class of general partition graphs. A graph $G= (V,E)$ is a
{\it general partition graph} if there exists a set $U$ and an assignment of non-empty subsets $U_x\subseteq U$ to the vertices of $G$ such
that two vertices $x$ and $y$ are adjacent if and only if $U_x\cap U_y\neq\emptyset$, and for every maximal stable set $S$ of $G$,
the set $\{U_x\,:\,x\in S\}$ is a partition of $U$. General partition graphs arise in the geometric setting of lattice polygon
 triangulations~\cite{MR1018966} and were studied in a series of
papers~\cite{MR1477536,MR1212882,MR960652,MR1038674,MR2007309,MR2080087,MR2794315}.

A {\it strong clique} in a graph $G$ is a clique (that is, a set of pairwise adjacent vertices) that meets all maximal stable sets.
McAvaney, Robertson and DeTemple proved in~\cite{MR1212874} that a graph $G$ is general partition if and only if
every edge of $G$ is contained in a strong clique. Together with some results from~\cite{MR1268776,MR2794315},
this result implies that every general partition graph is strongly equistable. As a possible attempt to settle the conjecture of Mahadev et al.,
Jim Orlin (personal communication, 2009) proposed a stronger conjecture stating that all equistable graphs are general partition
graphs (cf.~\cite{MR2794315}). The results of Peled and Rotics~\cite{MR2024275} and of Korach and Peled~\cite{MR2024271}, respectively, imply
that Orlin's conjecture holds within the classes of chordal graphs and of series-parallel graphs.
The conjecture has also been verified for simplicial graphs~\cite{MR3162288}, very
well-covered graphs~\cite{MR3162288}, line graphs~\cite{MR3162288}, EPT graphs~\cite{AlcGutKovMilRiz14}, AT-free graphs~\cite{MR2794315}, and various product graphs~\cite{MR2794315}.

As an intermediate conjecture, one that would follow from Orlin's conjecture and would imply the conjecture of Mahadev et al.,
Miklavi\v c and Milani\v c proposed in~\cite{MR2794315} {a} conjecture stating that every equistable graph has a strong clique.

To summarize, the classes of general partition graphs, strongly equistable graphs, and equistable graphs are related as follows:
$$\emph{general partition graphs} \subseteq \emph{strongly equistable graphs} \subseteq\emph{equistable graphs}\,.$$
{The conjecture of Mahadev et al.~states that the last inclusion is in fact an equality (which would also follow
from the conjecture due to Miklavi\v c and Milani\v c); the conjecture of Orlin states that both inclusions are equalities.}

In this paper, we prove that in fact all these three classes are pairwise distinct, thus disproving the conjectures of
Mahadev et al., of Orlin, and of Miklavi\v c and Milani\v c.
Our proofs rely on linear algebra, and all the separating examples are complements of line graphs of triangle-free graphs.
We present a sufficient condition for a graph $G$ such that the complement of its line graph
is equistable but not general partition, and construct several graphs satisfying this condition:
an infinite family based on circulant graphs, and a particular $14$-vertex graph, which is not strongly equistable.
We can thus refine the relations between the above mentioned graph classes as follows:
$$\emph{general partition graphs} \subsetneq \emph{strongly equistable graphs} \subsetneq\emph{equistable graphs}\,.$$
{We also
give an alternative proof to that of Mahadev et al.~\cite{MR1268776} of the fact that every strongly equistable graph is equistable.}

The paper is structured as follows. In Section~\ref{sec:equistarable} we define and study the class of equistarable graphs, the complements of
line graphs of which will form the basis for our constructions.
In Section~\ref{sec:sufficient}, we state and prove a sufficient condition for a graph to be equistarable.
Section~\ref{sec:counterexamples} contains our constructions.
In Section~\ref{sec:proof}, we give an alternative, more constructive proof than that given by
Mahadev et al.~\cite{MR1268776}, of the fact that every strongly equistable graph is equistable.

\bigskip
\noindent{\em Basic definitions and notation.}
The {\it complement} of a graph $G$ is the graph $\overline{G}$ with vertex set $V(G)$ in which two distinct vertices
are adjacent if and only if they are non-adjacent in $G$. Given two graphs $G$ and $H$, we say that a graph $G$ is {\it $H$-free}
if no induced subgraph of $G$ is isomorphic to $H$.
The {\it degree} of a vertex $v$ in a graph $G$ is the number of edges incident with $v$, denote by $d_G(v)$. By
$\delta(G)$ we denote the minimum degree of a vertex in $G$.
By $K_n$ we denote the complete graph with $n$ vertices.
For a positive integer $n$, we write $[n]$ for $\{1,\ldots, n\}$.

We say that a graph $G$ is a {\em line graph} if there exists a graph $H$ such that there is a bijective mapping
$\varphi: V(G)\to E(H)$ such that two distinct vertices $u,v\in V(G)$ are adjacent in $G$ if and only if
the edges $\varphi(u)$ and $\varphi(v)$ of $H$ have a common endpoint. If this is the case, then graph $H$ is
called the {\em root graph} of $G$, while $G=L(H)$ is the {\em line graph} of $H$. Except for $G= K_3$, the root of a
connected line graph $G$ is unique (up to isomorphism), and can be computed in linear time~\cite{MR0424435}.

\section{Equistarable graphs}\label{sec:equistarable}

Given a graph $G$ and a vertex $v\in V(G)$, the {\it star rooted at $v$} is the set $E(v)$ of all edges incident with $v$.
A {\it star} of $G$ is a star rooted at some vertex of $v$, and a star is said to be {\it maximal} if
it is not properly contained in another star of $G$.

\begin{lem}\label{lem:maximal_star}
Let $G$ be a graph and let $v\in V(G)$.
Then, the following are equivalent:
\begin{enumerate}
  \item $E(v)$ is a maximal star.
  \item Either $d_G(v)\ge 2$, or $d_G(v) = d_G(w) = 1$ where $w$ is the unique neighbor of $v$,
  or $d_G(v) = 0$ and $E(G) = \emptyset$.
\end{enumerate}
\end{lem}

\begin{proof}
Suppose that $d_G(v)\ge 2$, and let $x$ and $y$ be two distinct neighbors of $v$.
Then clearly $\{vx,vy\}\subseteq E(w)$ if and only if $w = v$, which implies that $E(v)$ is a maximal star.
Suppose now that $d_G(v)= 1$. Then,
$E(v)$ is properly contained in $E(w)$ where $w$ is the unique neighbor of $v$
if and only if $d_G(w) >1$. Clearly, if $x\in V(G)\setminus \{v,w\}$, then
$E(v)$ cannot be properly contained in $E(x)$.
This shows that $E(v)$ is a maximal star if and only if $d_G(w) = 1$.
Finally, if $v$ is an isolated vertex, then $E(v) = \emptyset$ is properly contained in another star if and only if
$G$ is edgeless.\end{proof}

We now introduce a class of graphs that will form the basis for our counterexamples.

\begin{defi}
A graph $G = (V,E)$ without isolated vertices is said to be {\em equistarable} if there exists
a mapping $\varphi:E\to\mathbb{R}^+$ on the edges of $G$
such that a subset $F\subseteq E$ is a maximal star in $G$ if and only if $\varphi(F) :=\sum_{e\in F}\varphi(e)=1$.
The mapping $\varphi$ is called an {\em equistarable weight function} of $G$.
\end{defi}

Observe that for every graph $G$, a subset of edges $F\subseteq E(G)$ forms a clique in $L(G)$
if and only if $F$ is either contained in a star of $G$, or $F$ is a triangle (the edge set of a $K_3$ in $G$).
Consequently, $F$ is a maximal clique in $L(G)$ if and only if it is a maximal element in the collection
of all triangles and maximal stars in $G$.

\begin{lem}\label{lem:triangle-free}
Let $G$ be a triangle-free graph.
Then $G$ is equistarable if and only if $\overline{L(G)}$ is equistable.
\end{lem}

\begin{proof}
Let $G = (V,E)$ be a triangle-free graph.
Since $G$ is triangle-free, a subset $C\subseteq E$ is a maximal clique in $L(G)$
if and only if $C$ is a maximal star in $G$.
Consequently, $S\subseteq E$ is a maximal stable set in
$\overline{L(G)}$, if and only if $S$ is a maximal star in $G$.
Thus, $G$ is equistarable if and only if $\overline{L(G)}$ is equistable.
\end{proof}

Given a graph $G$, a (maximal) clique in the complement of its line graph $\overline{L(G)}$
corresponds to a (maximal) stable set in $L(G)$, and hence to a (maximal) matching in $G$.
(A {\it matching} is a subset of pairwise disjoint edges, and it is {\it maximal} if it is not contained
in a larger matching.)
A matching in a graph $G$ is said to be {\it perfect} if it consists of $|V(G)|/2$ edges.

\begin{lem}\label{lem:strong-clique}
Let $G$ be a triangle-free graph with $\delta(G)\ge 2$, and let $M\subseteq E(G)$.
Then $M$ is a strong clique in $\overline{L(G)}$
if and only if $M$ is a perfect matching in $G$.
\end{lem}

\begin{proof}
Let $G = (V,E)$ be a triangle-free graph with $\delta(G)\ge 2$, and
let $M\subseteq E$.

As observed above, $M$ is a clique in
$\overline{L(G)}$ if only if it is a matching in $G$.
Moreover, the condition that $M$ is a strong clique in $\overline{L(G)}$
is equivalent to the condition that $M$ is a clique in $\overline{L(G)}$ intersecting
 all maximal stable sets in $\overline{L(G)}$, which is in turn equivalent to the
condition that $M$ is a stable set in
$L(G)$ intersecting all maximal cliques in $L(G)$.
Since $G$ is triangle-free,
$M$ is a stable set in
$L(G)$ intersecting all maximal cliques in $L(G)$,
if and only if
$M$ is a matching in $G$ intersecting all maximal stars in $G$.
Since $\delta(G)\ge 2$, Lemma~\ref{lem:maximal_star} implies that every star in $G$ is maximal.
Therefore,
$M$ is a matching in $G$ intersecting all maximal stars in $G$,
if and only if
$M$ is a matching in $G$ intersecting all stars in $G$, which is equivalent to $M$ being a perfect
matching.
\end{proof}

As we show next, the condition that $\overline{L(G)}$ is a general partition graph can be expressed in terms of $2$-extendability
of $G$. Recall that a graph $G$ is said to be {\it $k$-extendable} if it has a {\it $k$-matching} (that is, a matching of size $k$),
and every $k$-matching is contained in a perfect matching (see, e.g.,~\cite{MR583220,DBLP:journals/dm/Plummer94a}).

\begin{lem}\label{lem:extendable}
Let $G$ be a triangle-free graph with $\delta(G)\ge 2$ and containing a $2$-matching.
Then, $\overline{L(G)}$ is a general partition graph if and only if $G$ is $2$-extendable.
\end{lem}

\begin{proof}
Notice that two distinct vertices $e,f$ of
$\overline{L(G)}$, that is, $e,f\in V(\overline{L(G)}) = E(G)$,
are adjacent, if and only if $e$ and $f$ are non-adjacent in $L(G)$, if and only if
$e$ and $f$ form a matching of size $2$ in $G$.
Thus, edges of $\overline{L(G)}$ are in a natural bijective correspondence with $2$-matchings in $G$.

It is known that a graph is a general partition graph if and only if
every edge in it is contained in a strong clique~\cite{MR1212874}.
By Lemma~\ref{lem:strong-clique}, a subset $M\subseteq E(G)$ is a strong clique in $\overline{L(G)}$
if and only if $M$ is a perfect matching in $G$.

Therefore, the condition that $\overline{L(G)}$ is a general partition graph
is equivalent to the condition that every edge of $\overline{L(G)}$, which is
equivalent to the condition that every $2$-matching of $G$ is contained in a perfect matching.
Since $G$ is assumed to contain a $2$-matching, this latter condition is equivalent to the
condition that $G$ is $2$-extendable.
\end{proof}

Lemmas~\ref{lem:strong-clique} and \ref{lem:extendable} imply that
to disprove Orlin's conjecture,
it suffices to construct a graph $G$ having the following properties:
\begin{itemize}
  \item $G$ is triangle-free, of minimum degree at least $2$, and containing a $2$-matching.
  \item $G$ is equistarable,
  \item $G$ is not $2$-extendable.
\end{itemize}
{In the following section, we will describe a general sufficient condition for a graph on at least $5$ vertices
 to have all the stated properties,
except perhaps the triangle-freeness. We will also provide an infinite family of triangle-free graphs satisfying the condition.}

\section{A sufficient condition for equistarability}\label{sec:sufficient}

Given a cycle $C$ in a graph $G$, a {\it chord} of $C$ is a pair of adjacent vertices that appear non-consecutively on $C$.
Given a graph $G$, an odd cycle $C$ in $G$,
and two disjoint edges $e,e'\in E(C)$, the graph $C-\{e,e'\}$ consists of two paths,
say $P_0$ and $P_1$, exactly
one of which, say $P_0$, is of even length.
An {\it $(e,e')$-non-crossing even chord of $C$} is a chord of $C$
such that its two endpoints belong to $P_0$ and are at even distance on $P_0$.
An {\it $(e,e')$-crossing odd chord of $C$} is a chord of $C$
such that exactly one of its endpoints, say $v$, belongs to $P_0$, and
$d_{P_0}(v,e)\equiv d_{P_0}(v,e')\equiv 1\pmod 2$.

\begin{defi}\label{def:bad}
We say that a graph $G$ of odd order $n$
is {\em bad} if
$G$ has a Hamiltonian cycle $C_n$ such that
for every two disjoint edges $e,e'\in E(C_n)$,
$G$ contains either an $(e,e')$-non-crossing even chord of $C_n$, or
an $(e,e')$-crossing odd chord of $C_n$.
\end{defi}

{For instance, every complete graph of odd order $n\ge 3$ is bad (with respect to any of its Hamiltonian cycles). }
For our constructions, we will be interested in triangle-free bad graphs.
For every odd $n\ge 9$, there exists a triangle-free bad graph of order $n$.
We leave it as an exercise for the reader to verify that the $4$-regular circulant graphs obtained by taking an $n$-cycle $C_n$ with odd $n\ge 11$, and
adding an edge between two vertices if and only if they are at distance $3$ on the cycle, are bad
(in fact, for every two disjoint edges $e,e'\in E(C_n)$,
$G$ contains a $(e,e')$-crossing odd chord of $C_n$; see Fig.~\ref{fig:circulants} for examples of orders $11$, $13$, $15$).
A triangle-free bad graph of order $9$ will be analyzed in Section~\ref{sec:example}.

\begin{figure}[!ht]
  \begin{center}
\includegraphics[height=40mm]{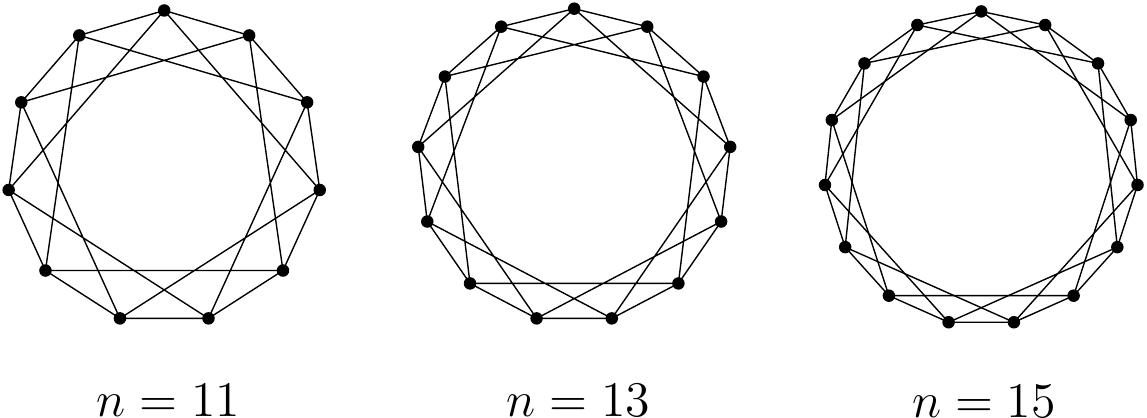}
  \end{center}
\caption{Examples of bad graphs.}
\label{fig:circulants}
\end{figure}

\begin{thm}\label{thm:bad-equistarable}
Every bad graph is equistarable.
\end{thm}

\begin{proof}
Let $G$ be a bad graph. By construction, the edge set of $G$ is partitioned into two disjoint sets, $E(G) = E(C_n)\cup F$,
where $C_n$ is a Hamiltonian cycle in $G$
satisfying the condition from Definition~\ref{def:bad}.

We will now construct an equistarable weight function $\varphi$ of $G$ in two steps, as follows.
Let $F = \{f_1,\ldots, f_r\}$,
let $\varepsilon \in (0,1/(3r))$ (the precise value of $\varepsilon$ will be determined later) and
let $\alpha_1,\ldots, \alpha_r\in (0,\varepsilon)$ be
positive real numbers algebraically independent over $\mathbb{Q}$, that is,
\begin{equation}\label{eq:independent}
\textrm{if~~$\sum_{i = 1}^rq_i\alpha_i = 0$ where $q_i\in  \mathbb{Q}$ for all $i\in [r]$,
then $q_1= \ldots = q_r = 0$}\,.
\end{equation}
For all $i\in [r]$, let $\varphi(f_i) = \alpha_i$.
It remains to assign the values to the $n$ edges in $E(C_n)$.
The function $\varphi$ should satisfy the constraints
of the form
$\sum_{e\in E(v)}\varphi(e) = 1$, for every vertex $v\in V(G)$.
Substituting into this system of equations the already fixed values
$\varphi(f_i) = \alpha_i$ for $f_i\in F$ yields
the following linear system with $n$ variables and $n$ equations.
Denoting by $(v_1,\ldots, v_n)$ the cyclic order of vertices on $C_n$,
the edges of $C_n$ by $e_j = v_{j-1}v_{j}$ (indices modulo $n$),
the equation corresponding to $v_j$ is
\begin{equation}\label{equation-v}
x_{j}+x_{j+1} = b_j\,,
\end{equation}
where
\begin{equation}\label{rhs}
  b_{j} = 1-\sum_{i\,:\,v_{j}\in f_i}\alpha_i\,.
\end{equation}
It can be seen that the system's coefficient matrix has determinant $2$, and hence
the system has a unique solution $\beta = (\beta_j\,,\,1\le j\le n)$.
In fact, the solution can be explicitly computed:
the $j$-th component of $\beta$ is given by
\begin{equation}\label{beta}
  \beta_{j} = \frac{1}{2}\sum_{k = 0}^{n-1}(-1)^kb_{j+k}
\end{equation}
(indices modulo $n$).
We complete the definition of the mapping $\varphi:E(G)\to\mathbb{R}_+$,
by setting $\varphi(e_j) = \beta_j$ for all $j\in [n]$.

It remains to prove that $\varphi$ is an equistarable weight function of $G$.
First, we argue that with an appropriate choice of $\varepsilon$, we can achieve $\varphi(e)\ge 0$ for all $e\in E(G)$.
By construction, we have $\varphi(f_i)>0$ for all $i\in [r]$.
Notice that in the limit when $\varepsilon\to 0$, we have $\alpha_i \to 0$ for all $i$.
Consequently, by \eqref{rhs} and \eqref{beta}, also $\beta_j\to 1/2$ for all $j$,
as $\varepsilon\to 0$.
Therefore,
for a small enough $\varepsilon \in (0,1/(3r))$,
we will have $\beta_j\in (1/3,2/3)$
for all $j\in [n]$.
In particular, we have $\varphi(e_j)>0$ for all $j\in [n]$.

By construction, for every maximal star $E'\subseteq E(G)$, we have $\varphi(E') = 1$.
We still need to show that if $E'\subseteq E(G)$ such that $\varphi(E') = 1$, then $E'$ is a star.
Let $\varphi(E') = 1$, and let $p = |E'\cap E(C_n)|$.
Recall that $1/3<\varphi(e)<2/3$ for all $e\in E'\cap E(C_n)$.
Thus, if $p\le 1$ then
$$\varphi(E')= \varphi(E'\cap E(C_n)) + \varphi(E'\cap F)< 2/3 + \sum_{i = 1}^r\alpha_i<2/3+r\varepsilon <1\,,$$ a contradiction.
Also, if $p\ge 3$ then
$$\varphi(E')\ge \varphi(E'\cap E(C_n)) > p/3\ge 1\,,$$ a contradiction.
So we have $p = 2$. Let $E'\cap E(C_n) = \{e,e'\}$.
If $e$ and $e'$ share a common endpoint, say $v_j$, then we may assume
$e = e_{j}$ and $e' = e_{j+1}$, and
$$\varphi(e_j)+\varphi(e_{j+1})+\varphi(E'\cap F) = \varphi(E') = 1 = \varphi(E(v_j)) = \varphi(e_j)+\varphi(e_{j+1})+ \varphi(E(v_j)\cap F)\,,$$
implying $\varphi(E'\cap F)=\varphi(E(v_j)\cap F)$ and hence, due to the algebraic independence of the
$\alpha_i$'s, we have $E'\cap F=E(v_j)\cap F$, and consequently $E' = (E(v_j)\cap F)\cup\{e_j,e_{j+1}\} = E(v_j)$ is a (maximal) star.
Suppose now that $e$ and $e'$ are disjoint.
Without loss of generality, we may assume that
$e = e_j$ and $e' = e_k$ such that $j<k$ and $k-j\ge 3$ is odd
(the other case is similar).
We have
$$
\beta_{j}+\beta_{k}=
\frac{1}{2}\left(\sum_{\ell = 0}^{n-1}(-1)^\ell b_{j+\ell}
+\sum_{\ell = 0}^{n-1}(-1)^\ell b_{k+\ell}\right)=
\frac{1}{2}\left((-1)^j\sum_{\ell = j}^{n+j-1}(-1)^\ell b_{\ell}
+(-1)^k\sum_{\ell = k}^{n+k-1}(-1)^\ell b_{\ell}\right)$$
and consequently
$$2(-1)^j(\beta_{j}+\beta_{k})=
\sum_{\ell = j}^{n+j-1}(-1)^\ell b_{\ell}-\sum_{\ell = k}^{n+k-1}(-1)^\ell b_{\ell}
$$$$=
\sum_{\ell = j}^{k-1}(-1)^\ell b_{\ell}-\sum_{\ell = n+j}^{n+k-1}(-1)^\ell b_{\ell}=
\sum_{\ell = j}^{k-1}(-1)^\ell b_{\ell}+\sum_{\ell = j}^{k-1}(-1)^\ell b_{\ell}=
2\sum_{\ell = j}^{k-1}(-1)^\ell b_{\ell}\,.$$
Therefore
\begin{equation}\label{beta2}
\beta_{j}+\beta_{k}=
(-1)^j\sum_{\ell = j}^{k-1}(-1)^\ell b_{\ell}=\sum_{\ell = 0}^{k-j-1}(-1)^\ell b_{j+\ell}\,.
\end{equation}
Substituting equations~\eqref{rhs} and~\eqref{beta2} into
$$\beta_{j}+\beta_{k} +\varphi(E'\cap F) =  1\,,$$
observing that $\sum_{\ell = 0}^{k-j-1}(-1)^\ell = 1$ and simplifying,
we obtain
$$\sum_{\ell = 0}^{k-j-1}\sum_{i\,:\,v_{j+\ell}\in f_i}(-1)^\ell \alpha_i= \sum_{f_i\in E'\cap F}\alpha_i\,.$$
Equivalently:
$$\sum_{\ell = 1}^r\left(\sum_{\ell = 0}^{k-j-1}(-1)^\ell\mathbf{1}_{\{v_{j+\ell}\in f_i\}}\right)\alpha_i=
\sum_{\ell = 1}^r\mathbf{1}_{\{f_i\in E'\cap F\}}\alpha_i\,,$$
where $\mathbf{1}_{\{X\}}$ is $1$ if $X$ is true and $0$ otherwise.
By the algebraic independence of the $\alpha_i$'s,
we have
$$\sum_{\ell = 0}^{k-j-1}(-1)^\ell\mathbf{1}_{\{v_{j+\ell}\in f_i\}}= \mathbf{1}_{\{f_i\in E'\cap F\}}$$
for all $i\in [r]$, which implies:
\begin{enumerate}[(i)]
  \item for every $f_i\in E'\cap F$, exactly one endpoint of
  $f_i$ is in the set $\{v_j,\ldots, v_{k-1}\}$, moreover
  this endpoint is of the form $v_{j+\ell}$ such that $\ell$ is even;
  \item for every $f_i\in F\setminus E'$, either both endpoints of $f_i$ are outside the set $\{v_j,\ldots, v_{k-1}\}$,
  or they are both in, say $f_i = v_{j+\ell_1}v_{j+\ell_2}$, and $\ell_1\not\equiv \ell_2\pmod 2$.
\end{enumerate}
Since the edges $e_j$ and $e_k$ are disjoint, we can apply the fact that
$C_n$ satisfies the condition from Definition~\ref{def:bad}.
We consider two cases.

{\it Case 1. There exists an $(e_j,e_k)$-non-crossing even chord of $C_n$.}
Then such a chord is of the form $f_i=v_{j+\ell_1}v_{j+\ell_2}$ where $i\in [r]$, $\ell_1,\ell_2\in \{0,\ldots, k-j-1\}$
and $\ell_1\equiv \ell_2\pmod 2$.
By $(i)$, we have $f_i\not\in E'$. But now $f_i\in F\setminus E'$, with both endpoints in the set $\{v_j,\ldots, v_{k-1}\}$, contrary to
$(ii)$.

{\it Case 2. There exists an $(e_j,e_k)$-crossing odd chord of $C_n$.}
Let $P$ denote the path $(v_j, v_{j+1}, \ldots, v_{k-1})$.
In this case, there exists a vertex $v\in V(P)$ such that
$d_{P}(v,e_j)\equiv d_{P}(v,e_k)\equiv 1 \pmod 2$, and there exists an edge $f_i\in F$ connecting $v$
to a vertex in the component of $C_n-\{e_j,e_k\}$ not containing $v$. It can be seen that this implies that
$v = v_{j+\ell}$ for some $\ell\in \{0,\ldots, k-j-1\}$ with $\ell\equiv 1\pmod 2$, and that $f_i$ has exactly one endpoint in the set
$\{v_j,\ldots, v_{k-1}\}$. By $(ii)$, we have $f_i\not\in F\setminus E'$, and therefore $f_i\in F\cap E'$. But now, we have the contradiction with the assumption that $\ell$ is odd.

This completes the proof that $\varphi$ is an equistarable weight function of $G$.
\end{proof}

\section{Counterexamples}\label{sec:counterexamples}


The following is an easy consequence of Theorem~\ref{thm:bad-equistarable}.

\begin{prop}\label{cor:cntrex}
Let $G$ be a triangle-free bad graph. Then, $\overline{L(G)}$ is an equistable graph without a strong clique.
In particular, $\overline{L(G)}$ is not a general partition graph.
\end{prop}

\begin{proof}
By construction, $G$ is triangle-free, of minimum degree at least $2$, and
contains a matching of size $2$.
Since $G$ is of odd order, it does not have a perfect matching.
Note that since $G$ has a Hamiltonian cycle, it satisfies the condition $\delta(G)\ge 2$.
By Lemma~\ref{lem:strong-clique},
$\overline{L(G)}$ does not have a strong clique.
By Lemma~\ref{lem:triangle-free}, it remains to show that $G$ is equistarable.
But this { follows from} Theorem~\ref{thm:bad-equistarable}.
\end{proof}

Since there exist triangle-free bad graphs (for example, the circulants mentioned in Section~\ref{sec:sufficient}),
Proposition~\ref{cor:cntrex} disproves the conjectures of Orlin and of Miklavi\v c and Milani\v c.
In the next two subsections, we will refine the fact that not every equistable graph is a general partition graph by showing that
there exist equistable graphs that are not strongly equistable, as well as
strongly equistable graphs that are not general partition graphs.

\subsection{Not all equistable graphs are strongly equistable}\label{sec:example}

The definition of strongly equistable graphs motivates the following definition.

\begin{defi}
For a graph $G$, we denote by ${\cal S^*}(G)$ the set of all maximal stars of $G$, and by ${\cal T^*}(G)$ the set of all other nonempty subsets of $E(G)$.
A graph $G = (V,E)$ without isolated vertices is said to be {\em strongly equistarable} if
for each $T\in {\cal T}^*(G)$ and each $\gamma\le 1$ there exists a function
$\varphi:E\to \R_+$ such that  $\varphi(S) = 1{\rm\ for\ all\ } S\in {\cal S}^*(G)$, and $\varphi(T)\neq \gamma$.
\end{defi}

The following observation can be proved similarly as Lemma~\ref{lem:triangle-free}.

\begin{lem}\label{lem:triangle-free-2}
Let $G$ be a triangle-free graph.
Then $G$ is strongly equistarable if and only if $\overline{L(G)}$ is strongly equistable.\hfill\qed
\end{lem}

Let $G^*$ be the graph depicted in Fig.~\ref{fig:example}.

\begin{figure}[h!]
  \begin{center}
\includegraphics[height=40mm]{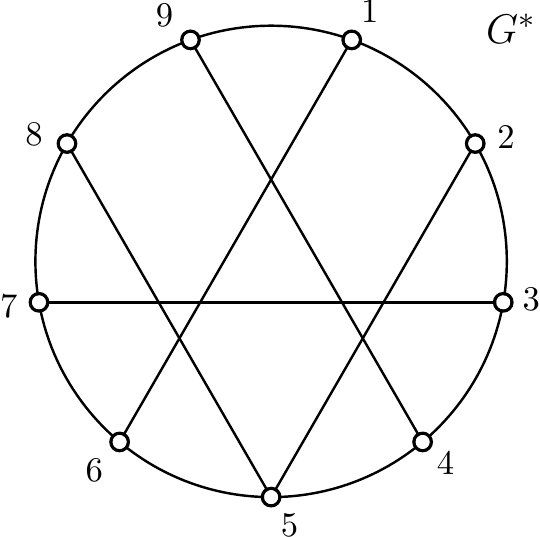}
  \end{center}
\caption{A $9$-vertex, $14$-edge bad graph.}
\label{fig:example}
\end{figure}

\begin{prop}\label{prop:G*}
Graph $G^*$ is equistarable but not strongly equistarable.
\end{prop}

\begin{proof}
Let $G^* = (V,E)$ where the vertices are $V = \{1,\ldots, 9\}$ (see Fig.~\ref{fig:example}). Clearly, $G^*$ is a triangle-free graph.

We claim that $G^*$ is a bad graph. To see this, consider the Hamiltonian cycle $C = (1,2,\ldots, 9,1)$ in $G^*$.
For two disjoint edges $e,e'\in E(C)$, we say that an edge $f\in E(G^*)\setminus E(C)$ is a {\it witness for
$e$ and $e'$} if $f$ is either an $(e,e')$-non-crossing even chord of $C$, or
an $(e,e')$-crossing odd chord of $C$.
In Table~\ref{table}, we give witnesses for certain pairs of disjoint edges of $C$.
Each of the remaining pairs is easily checked to be symmetric to one of the pairs in the table.
This shows that $G^*$ is bad.

\begin{table}[h!]
\begin{center}
\begin{tabular}{|cc|cc|cc|}
  \hline
  \rule{0.3cm}{0cm}$\{e,e'\}$\rule{0.3cm}{0cm} & witness & \rule{0.3cm}{0cm}$\{e,e'\}$\rule{0.3cm}{0cm} & witness & \rule{0.3cm}{0cm}$\{e,e'\}$\rule{0.3cm}{0cm} & witness\\
  \hline
$\{19,23\}$ & $16$ & $\{12,78\}$ & $49$ & $\{34,56\}$ & $25$\\
$\{19,34\}$ & $25$ & $\{12,89\}$ & $37$ & $\{34,67\}$ & $25$\\
$\{19,45\}$ & $16$ & $\{23,45\}$ & $25$ & $\{34,78\}$ & $25$\\
$\{12,34\}$ & $25$ & $\{23,56\}$ & $49$ & $\{34,89\}$ & $25$\\
$\{12,45\}$ & $37$ & $\{23,67\}$ & $16$ & $\{45,67\}$ & $16$\\
$\{12,56\}$ & $37$ & $\{23,78\}$ & $49$ & $\{45,78\}$ & $16$\\
$\{12,67\}$ & $37$ & $\{23,89\}$ & $16$ & $\{45,89\}$ & $16$\\
 \hline
\end{tabular}
\caption{Checking badness of $G^*$\label{table}}
\end{center}
\end{table}

By Theorem~\ref{thm:bad-equistarable}, $G^*$ is equistarable.
We assume the notation from the proof of Theorem~\ref{thm:bad-equistarable} (when specified to $G^*$), in particular,
$C=C_9$ denotes the Hamiltonian cycle $(1,\ldots, 9,1)$, and $F = \{f_1,\ldots, f_5\}$ denotes the set of edges of $G^*$ not in $C$.

Let $T = \{\{1,9\},\{3,7\}\}\in {\cal T}^*(G^*)$. To show that $G^*$ is not strongly equistarable, we will prove that
for every  $\varphi:E\to \R_+$ such that  $\varphi(S) = 1{\rm\ for\ all\ } S\in {\cal S}^*(G^*)$, we have
$\varphi(T)= 1/2$.

Every mapping $\varphi:E\to \R_+$ such that  $\varphi(S) = 1{\rm\ for\ all\ } S\in {\cal S}^*(G^*)$
naturally corresponds to a solution $x\in \mathbb{R}^E$ of the linear system $Ax = \mathbf{1}$, where
$A\in \{0,1\}^{V\times E}$ is the incidence matrix of $G^*$ (having $a_{v,e} = 1$ if and only if $v$ is an endpoint of $e$),
and $\mathbf{1}\in \mathbb{R}^V$ is the all-one vector of length $9$.
Since $G^*$ is equistarable, the system has a solution.

It is known that the kernel of the incidence matrix of a connected non-bipartite graph $G$
is of dimension $|E(G)|-|V(G)|$ (see, e.g., \cite[Corollary 5.3]{MR1306983}).
In the case of $G^*$, this implies that the kernel of $A$ is of dimension $5$.
A basis for the kernel of $A$ can be obtained as follows.
Each of the $5$ edges $f_i\in F$
forms a unique
even cycle $C_i$
together with a subpath of $C$
(in particular, $E(C_i)\cap F = \{f_i\}$). Fixing a cyclic ordering $e^i_1,\ldots, e^i_{2n_i}$ of the
edges of $C_i$, let the vector $x^i\in \mathbb{R}^E$ be defined as
$$x^i_e = \left\{
           \begin{array}{ll}
             (-1)^j, & \hbox{if $e = e^i_j$ for some $j\in \{1,\ldots, 2n_i\}$} \\
            0, & \hbox{otherwise.}
           \end{array}
         \right.
$$
It can be seen that the vectors $x^1,\ldots, x^5$ are linearly independent
and that $Ax^i = 0$ for all $i\in \{1,\ldots, 5\}$. Therefore,
$\{x^1,\ldots, x^5\}$ is a basis of the kernel of $A$.
(In fact, this is a special case of a more general construction, see~\cite[p.~$302$]{MR1306983}.)

It follows that every solution $x\in \mathbb{R}^E$ of the system $Ax = \mathbf{1}$ is of the form
$$x = x^0+\sum_{i = 1}^5 \lambda_i x^i\quad {\textrm{for some }}\lambda_1,\ldots, \lambda_5\in \mathbb{R}\,,$$
where $x^0$ is any particular solution to $Ax = \mathbf{1}$.
Let us take for $x^0$ the vector in
$\mathbb{R}^E$ defined by
$$x^0_e = \left\{
            \begin{array}{ll}
              1/2, & \hbox{if $e\in C$} \\
              0, & \hbox{otherwise.}
            \end{array}
          \right.
$$

Now, let $\varphi:E\to \R_+$ be an arbitrary mapping with $\varphi(S) = 1{\rm\ for\ all\ } S\in {\cal S}^*(G^*)$.
Identifying $\varphi$ with a vector in $\R^E$, we can write
$$\varphi = x^0+\sum_{i = 1}^5 \lambda_i x^i$$
for some $\lambda_1,\ldots, \lambda_5\in \mathbb{R}$.
We will now show that for the set $T = \{e,f\}\in {\cal T}^*(G^*)$,
where $e = \{1,9\}$ and $f = \{3,7\}$,
we have  $\varphi(T) = 1/2$.
Without loss of generality, we can assume that $f = f_5$.
Observe that the edge $e$ belongs to $C_5$ but not to any of $C_1,\ldots, C_4$. The same of course holds for the edge $f$.
This implies that $x^i_e = x^i_f = 0$ for all $i\in \{1,\ldots, 4\}$.
Moreover, edges $e$ and $f$ are at distance $2$ on $C_5$, which implies
$x^5_{e}+x^5_{f} = 0$. By the definition of $x^0$, we have $x^0_e = 1/2$ and $x^0_f = 0$.
Putting it all together,  we have
\begin{eqnarray*}
\varphi(\{e,f\}) &=&
x^0_{e}+\sum_{i = 1}^5 \lambda_i x^i_{e}
+x^0_{f}+\sum_{i = 1}^5 \lambda_i x^i_{f}\\
&=& \frac{1}{2} + \lambda_5(x^5_{e}+x^5_{f})= \frac{1}{2}\,.
\end{eqnarray*}
\end{proof}

\begin{cor}
The graph $\overline{L(G^*)}$ is equistable but not strongly equistable. In particular, it disproves the conjecture of Mahadev et al.
\end{cor}

The line graph of $G^*$ is depicted in Fig.~\ref{fig:line}.

\begin{figure}[!ht]
  \begin{center}
\includegraphics[height=50mm]{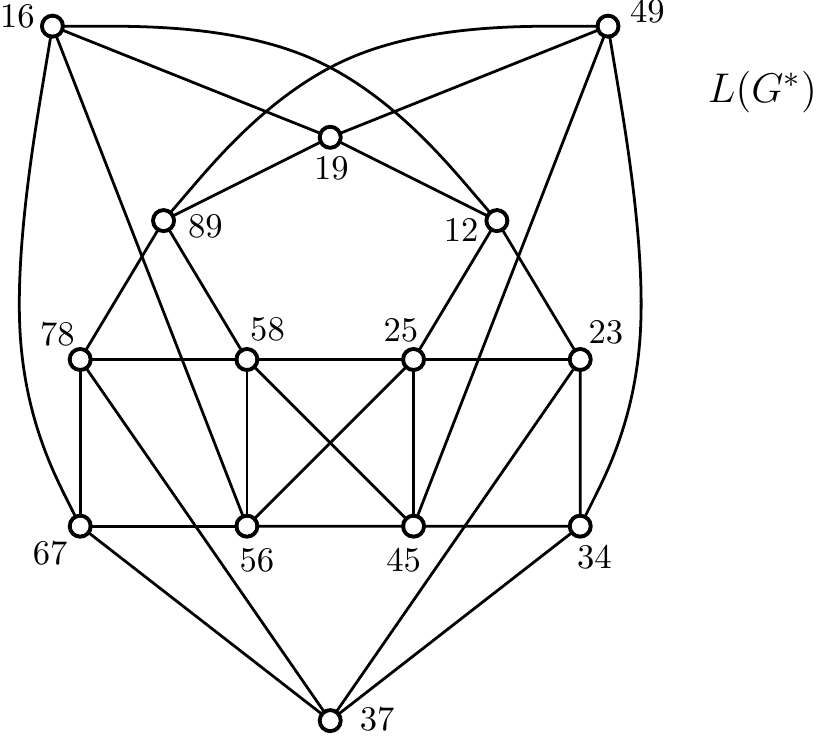}
  \end{center}
\caption{A $14$-vertex graph whose complement is equistable but not strongly equistable.}
\label{fig:line}
\end{figure}


\subsection{Not all strongly equistable graphs are general partition graphs}

\begin{prop}\label{prop:C_11_3}
The circulant $G = C_{11}(\{1,3\})$ is strongly equistarable.
\end{prop}

\begin{proof}
We proceed similarly as in the proof of Proposition~\ref{prop:G*}.
Let $G = C_{11}(\{1,3\})$. We assume that $V = V(G) = \{0,1,\ldots, 10\}$ and $E = E(G)= \{\{i,j\}\,:\,0\le i<j\le 10,~j-i\in \{1,3\} \textrm{ or } 11+i-j\in \{1,3\}\}$.
Recall that $G$ is a triangle-free bad graph. By Theorem~\ref{thm:bad-equistarable}, $G$ is equistarable.

We will verify that $G$ is strongly equistarable by definition.
Let $T \in {\cal T}^*(G)$ and $\gamma\le 1$.
We would like to show that $G$ admits a mapping $\varphi:E\to\R_+$
such that $\varphi(S) = 1{\rm\ for\ all\ } S\in {\cal S}^*(G)$ and $\varphi(T)\neq\gamma$.
We will refer to such a mapping as a {\it good weighting of $G$}.

If $T$ contains a star (equivalently: a maximal star) of $G$, then taking any equistarable weight function $\varphi$ of $G$
(for example, one constructed as in the proof of Theorem~\ref{thm:bad-equistarable}), we have
$\varphi(T)>1$ and hence $\varphi(T)\neq \gamma$.
So we may assume that $T$ does not contain any star of $G$.

Let $\varphi$ be an equistarable weight function of $G$ as constructed in the proof of Theorem~\ref{thm:bad-equistarable}. We
may assume that $\varphi(T) = \gamma$ (otherwise we are done).
Denoting by $C$ the Hamiltonian cycle $(0,1,\ldots, 10,0)$, and
proceeding as in the proof of Proposition~\ref{prop:G*}, we infer that
 $\varphi$
is of the form
$$\varphi = x^0+\sum_{i = 1}^{11}\lambda_i x^i$$
for some $\lambda_1,\ldots, \lambda_{11}\in \mathbb{R}$,
where
$x^0\in \mathbb{R}_+^E$ is the vector taking value $1/2$ on the edges $C$ (and $0$ on all other edges),
and
$x^i\in \{0,1,-1\}^{E}$ are the vectors taking values $1$ and $-1$ alternatingly
on the edges of an even cycle containing a unique chord of $C$ (and $0$ on all other edges).

{Suppose first that $T\cap E(C)= \emptyset$.
By the rotational symmetry of $G$, we can assume that $\{0,3\}\in T$. But now, for small enough $\epsilon>0$,
the function $\varphi':E\to\mathbb{R}$, defined by
$$\varphi'(e) = \left\{
                  \begin{array}{ll}
                    \varphi(e)+\epsilon, & \hbox{if $e\in \{\{0,1\},\{2,3\}\}$;} \\
                    \varphi(e)-\epsilon, & \hbox{if $e\in \{\{0,3\},\{1,2\}\}$;} \\
                    \varphi'(e), & \hbox{otherwise.}
                  \end{array}
                \right.$$
is a good weighting of $G$.
The function $\varphi'$ was obtained from $\varphi$ by adding a small nonzero multiple of one of the
$x^i$'s (namely, of the one corresponding to the cycle $(0,1,2,3,0)$). Any
such modification of $\varphi$ will be called an {\it $e$-perturbation}
where $e$ is the unique chord of $C$ such that $x^i_e\neq 0$.


From now on, assume that $T\cap E(C)\neq  \emptyset$.
If $E(C)\subseteq T$, then, since $T$ contains no star, there exists a chord $e$ of $C$ such that $e\not\in T$, and
an $e$-perturbation of $\varphi$ will result in a good weighting of $G$.}
Suppose now that $E(C)\nsubseteq T$.  Let $P$ be a longest subpath of $C$ all of whose edges are in $T$.
Due to the symmetry of $G$, we may assume that $0$ is an endpoint of $P$ and that
$\{0,10\}\not\in E(P)$.
If $\{2,10\}\in T$, then a $\{2,10\}$-perturbation
of $\varphi$ will result in a good weighting of $G$.
So we may assume that $\{2,10\}\not\in T$. Using the same perturbation, we can see that
$\{1,2\}\in T$.
Now, to avoid a good weighting of $G$ obtained from applying a $\{1,9\}$-perturbation to $\varphi$, we infer that
$\{1,9\}\in T$.
Recall that by the construction of $\varphi$, we have $\varphi(e)>1/3$ for all $e\in E(C)$.
Therefore,
$$\varphi(T)\ge \varphi(\{0,1\}) +\varphi(\{1,2\}) +
\varphi(\{1,9\}) >1\,,$$
which implies that $\varphi(T)\neq \gamma$.
\end{proof}

Propositions~\ref{cor:cntrex} and~\ref{prop:C_11_3} imply the following.

\begin{cor}
The graph $\overline{L(C_{11}(\{1,3\}))}$ is
a strongly equistable graph that is not a general partition graph.
\end{cor}

\section{An alternative proof of the fact that every strongly equistable graph is equistable}\label{sec:proof}

The fact that every strongly equistable graph is equistable was proved in~\cite{MR1268776}
using a dimension argument. Here we give an alternative, more constructive proof of this fact.

\begin{thm}[Mahadev et al.~\cite{MR1268776}]\label{thm:eq-seq}
Every strongly equistable graph is equistable.
\end{thm}

\begin{proof}
Let $G$ be a strongly equistable graph on $n$ vertices, and for each
set $T\in {\cal T}(G)$, fix a function
$\varphi_T:V(G)\to \R_+$ such that
$\varphi_T(S) = 1{\rm\ for\ all\ } S\in {\cal S}(G)$, and $\varphi_T(T)\neq 1$.

Let $\varphi:V(G)\to \R_+$ be a function satisfying
$\varphi(S) = 1{\rm\ for\ all\ } S\in {\cal S}(G)$
and minimizing the number, denoted by $t(\varphi)$, of sets $T\in {\cal T}(G)$ with
$\varphi(T)= 1$.
We claim that $t(\varphi) = 0$, that is, that $\varphi$ is an equistable weight function of $G$.

Suppose this is not the case, and let $T^*\in {\cal T}(G)$ be a set such that $\varphi(T^*)= 1$.
Note that $\varphi_{T^*}(T^*)\le |T^*|\le n$.
Fix any $\epsilon\in (0,1)$ such that
if $\varphi(T)\neq 1$ for some $T\in {\cal T}(G)$, then
$|\varphi(T)-1|>\epsilon$ and
$\varphi_{T^*}(T)>\epsilon$.
(Note that such an $\epsilon$ exists.)
Define a function $\varphi':V(G)\to \R_+$ by the rule
$\varphi'(x) = (1-\epsilon/n)\varphi(x) + (\epsilon/n)\varphi_{T^*}(x)$ for all $x\in V(G)$.

If $S\in {\cal S}(G)$, then
$$\varphi'(S) = (1-\epsilon/n) \varphi(S) + (\epsilon/n) \varphi_{T^*}(S) = 1\,.$$
If $T\in {\cal T}(G)$ such that $\varphi(T) <1$,
then  $\varphi(T)<1-\epsilon$ and therefore
$$\varphi'(T) = (1-\epsilon/n) \varphi(T) + (\epsilon/n) \varphi_{T^*}(T)
<(1-\epsilon/n)(1-\epsilon) + \epsilon = 1-(\epsilon/n)(1-\epsilon)<1\,.$$
If $T\in {\cal T}(G)$ such that $\varphi(T) >1$,
then  $\varphi(T)>1+\epsilon$ and therefore
$$\varphi'(T) = (1-\epsilon/n) \varphi(T) + (\epsilon/n) \varphi_{T^*}(T)
>(1-\epsilon/n)((1+\epsilon)+\epsilon^2/n =1+\epsilon(1-1/n)\ge 1\,.$$
Moreover,
$$\varphi'(T^*) = (1-\epsilon) \varphi(T^*) + \epsilon \varphi_{T^*}(T^*) =
1-\epsilon+ \epsilon \varphi_{T^*}(T^*) \neq 1\,.$$
Hence, $\varphi'$ satisfies $\varphi(S) = 1{\rm\ for\ all\ } S\in {\cal S}(G)$ and
$\varphi(T)\neq 1$ for all $T\in {\cal T}(G)$ such that $\varphi(T)\neq 1$ or $T = T^*$.
This implies
$t(\varphi')<t(\varphi)$, contrary to the definition of $\varphi$.
\end{proof}

From the above proof, an equistable weight function of a given strongly equistable
graph $G$ can be easily derived, assuming that for each
set $T\in {\cal T}(G)$ we have a function $\varphi_T:V(G)\to \R_+$ such that
$\varphi_T(S) = 1{\rm\ for\ all\ } S\in {\cal S}(G)$, and $\varphi_T(T)\neq 1$.

\section{Open problems}

Since bipartite graphs are triangle-free and Orlin's conjecture fails for complements of line graphs of triangle-free graphs,
the following question and its generalization seem natural.

\begin{problem}
Does Orlin's conjecture hold within the class of complements of line graphs of bipartite graphs?
\end{problem}


\begin{problem}
Does Orlin's conjecture hold within the class of perfect graphs?
\end{problem}

Note that the weaker conjecture by Mahadev et al.~(stating that every equistable graph is strongly equistable) holds for the class of perfect graphs, as shown in~\cite{MR1268776}.

Finally, let us remark that, to the best of our knowledge, the other
conjecture posed by Mahadev et al.~in~\cite{MR1268776},
stating that the strongly equistable graphs are closed under substitution, is still open.

\subsection*{Acknowledgements}

We would like to thank St\'ephan Thomass\'e for helpful discussions, in particular for suggesting the proof of Theorem~\ref{thm:eq-seq}
presented in Section~\ref{sec:proof}.

\bibliography{equistable-bib}{}
\bibliographystyle{plain}

\end{document}